\documentclass{article}
\title{Derived rules for predicative set theory: \\ an application of sheaves}
\author{Benno van den Berg \& Ieke Moerdijk}
\date{November 16, 2011}
\usepackage{ast2}
\newcommand{\N}{\mathbb{N}}

\begin{document}

\maketitle

\begin{abstract}
\noindent
We show how one may establish proof-theoretic results for constructive Zermelo-Fraenkel set theory, such as the compactness rule for Cantor space and the Bar Induction rule for Baire space, by constructing sheaf models and using their preservation properties.
\end{abstract}

\section{Introduction}

This paper is concerned with Aczel's predicative constructive set theory {\bf CZF} and with related systems for predicative algebraic set theory; it also studies extensions of {\bf CZF}, for example by the axiom of countable choice.

We are particularly interested in certain statements about Cantor space $2^\N$, Baire space $\N^\N$ and the unit interval $[0,1]$ of Dedekind real numbers in such theories, namely the compactness of $2^\N$ and of $[0,1]$, and the  related ``Bar Induction'' property for Baire space.  The latter property states that if  $S$ is a set of finite sequences of natural numbers for which
\begin{itemize}
\item[-] for each $\alpha$ there is an $n$ such that  $<\alpha(0),\alpha(1),....,\alpha(n)>$ belongs to $S$ (``$S$ is a bar''),
\item[-] if $u$ belongs to $S$ then so does every extension of $u$ (``$S$ is monotone''),
\item[-] if $u$ is a finite sequence for which the concatenation $u*n$ belongs to $S$ for all $n$, then $u$ belongs to $S$ (``$S$ is inductive''),
\end{itemize}
then the empty sequence $< \, >$ belongs to $S$. It is well-known that these statements, compactness of $2^\N$ and of $[0,1]$ and Bar Induction for $\N^\N$, cannot be derived in intuitionistic set or type theories. In fact, they fail in sheaf models over locales, as explained in \cite{fourmanhyland79}.  Sheaf models can also be used to show that all implications in the chain
\[ (BI) \Longrightarrow (FT) \Longrightarrow (HB) \]
are strict (where BI stands for Bar Induction for $\N^\N$, while FT stands for the Fan Theorem (compactness of $2^\N$) and HB stands for the Heine-Borel Theorem (compactness of the unit interval), see \cite{moerdijk84}).

On the other hand, one may also define Cantor space $\mathbf{C}$, Baire space $\mathbf{B}$, and the unit interval $\mathbf{I}$ as locales or formal spaces. Compactness is provable for formal Cantor space, as is Bar Induction for formal Baire space.  Although Bar Induction may seem to be a statement of a slightly different nature, it is completely analogous to compactness, as explained in \cite{fourmanhyland79} as well. Indeed, the locales $\mathbf{C}$ and  $\mathbf{I}$ have enough points (i.e., are true topological spaces) iff the spaces $2^\N$ and $[0,1]$ are compact, while the locale $\mathbf{B}$ has enough points iff Bar Induction holds for the space $\N^\N$. The goal of this paper is to prove that the compactness properties of these (topological) spaces do hold for ${\bf CZF}$ (with countable choice), however, when they are reformulated as derived rules. Thus, for example, Cantor space is compact in the sense that if $S$ is a property of finite sequences of 0's and 1's which is definable in the language of set theory and for which ${\bf CZF}$ proves
\begin{quote}
for all $\alpha$ in $2^\N$ there is an $n$ such that $<\alpha(0),\alpha(1),....,\alpha(n)>$  belongs to $S$ (``$S$ is a cover''),
\end{quote}
then there are such finite sequences $u_1,..., u_k$ for which ${\bf CZF}$ proves that each $u_i$ belongs to $S$ as well as that for each $\alpha$ as above there are an $n$ and an $i$ such that  $<\alpha(0),\alpha(1),....,\alpha(n)> = u_i$. We will also show that compactness of the unit interval and Bar Induction hold when formulated as derived rules for {\bf CZF} and suitable extensions of {\bf CZF}, respectively.

This is a proof-theoretic result, which we will derive by purely model-theoretic means, using sheaf models for ${\bf CZF}$ and a doubling construction for locales originating with Joyal. Although our results for the particular theory ${\bf CZF}$ seem to be new, similar results occur in the literature for other constructive systems, and are proved by various methods, such as purely proof-theoretic methods, realizability methods or our sheaf-theoretic methods.\footnote{For example, Beeson in \cite{beeson77} used a mixture of forcing and realizability for Feferman-style systems for explicit mathematics. Hayashi used proof-theoretic methods for {\bf HAH}, the system for higher-order Heyting arithmetic corresponding to the theory of elementary toposes in \cite{hayashi80}, and sheaf-theoretic methods in \cite{hayashi82} for the impredicative set theory {\bf IZF}, an intuitionistic version of Zermelo-Fraenkel set theory. Grayson \cite{grayson84} gives a sheaf-theoretic proof of a local continuity rule for the system {\bf HAH}, and mentions in \cite{grayson83} that the method should also apply to systems without powerset.} In this context it is important to observe that derived rules of the kind ``if $T$ proves $\varphi$, then $T$ proves $\psi$'' are different results for different $T$, and can be related only in the presence of conservativity results. For example, a result for {\bf CZF} like the ones above does not imply a similar result for the extension of {\bf CZF} with countable choice, or vice versa.

Our motivation to give detailed proofs of several derived rules comes from various sources. First of all, the related results just mentioned predate the theory {\bf CZF}, which is now considered as one of the most robust axiomatisations of predicative constructive set theory and is closely related to  Martin-L\"of type theory \cite{aczel78, aczel82, aczel86}. Secondly, the theory of sheaf models for {\bf CZF} has only recently been firmly established (see \cite{gambino02, gambino06, bergmoerdijk10b}), partly in order to make applications to proof theory such as the ones exposed in this paper possible. Thirdly, the particular sheaf models over locales necessary for our application hinge on some subtle properties and constructions of locales (or formal spaces) in the predicative context, such as the inductive definition of covers in formal Baire space in the absence of power sets. These aspects of predicative locale theory have only recently emerged in the literature \cite{coquandetal03, aczel06}. In these references, the regular extension axiom {\bf REA} plays an important role. In fact, one needs an extension of {\bf CZF}, which on the one hand is sufficiently strong to handle suitable inductive definitions, while on the other hand it is stable under sheaf extensions. One possible choice is the extension of {\bf CZF} by the smallness axiom for W-types and the axiom of multiple choice {\bf AMC} (see \cite{bergmoerdijk10}).

The results of this paper were presented by the authors on various occasions: by the second author in July 2009 at the TACL'2009 conference in Amsterdam and in March 2010 in the logic seminar in Manchester and by the first author in May 2010 at the meeting ``Set theory: classical and constructive'', again in Amsterdam. We would like to thank the organizers of all these events for giving us these opportunities. We are also grateful to the referee for a very careful reading of the original manuscript and to the editors for their patience.

\section{Constructive set theory}

Throughout the paper we work in Aczel's constructive set theory {\bf CZF}, or extensions thereof. (An excellent reference for {\bf CZF} is \cite{aczelrathjen01}.)

\subsection{CZF}

{\bf CZF} is a set theory whose underlying logic is intuitionistic and whose axioms are:
\begin{description}
\item[Extensionality:] $\forall x \, ( \, x \in a \leftrightarrow x \in b \, ) \rightarrow a = b$.
\item[Empty set:] $\exists x  \, \forall y  \, \lnot y \in x $.
\item[Pairing:] $\exists x \, \forall y \, (\,  y \in x \leftrightarrow y = a \lor y = b \, )$.
\item[Union:] $\exists x  \, \forall y \, ( \, y \in x \leftrightarrow \exists z \in a \, y \in z  \, )$.
\item[Set induction:] $\forall x \, \big(\forall y  \in x \, \varphi(y) \rightarrow \varphi(x)\big) \rightarrow \forall x \, \varphi(x)$.
\item[Infinity:] $\exists a \, \big(  \, ( \, \exists  x \, x \in a \, ) \land ( \, \forall x  \in a \, \exists y \in a \, x \in y \, ) \big)$.
\item[Bounded separation:] $\exists x \,  \forall y \, \big( \, y \in x \leftrightarrow y \in a \land \varphi(y) \, \big) $, for any bounded formula $\varphi$ in which $a$ does not occur.
\item[Strong collection:] $\forall x \in a \, \exists y \, \varphi(x,y) \rightarrow \exists b \, \mbox{B}(x \in a, y \in b) \, \varphi$.
\item[Subset collection:] $\exists c \, \forall z \, \big( \, \forall x \in a \, \exists y \in b \, \varphi(x,y,z) \rightarrow \exists d \in c \, \mbox{B}(x \in a, y \in d) \, \varphi(x, y, z) \, \big) $.
\end{description}
In the last two axioms, the expression
\[ \mbox{B}(x \in a, y \in b) \, \varphi. \]
has been used as an abbreviation for $\forall x \in a \, \exists y \in b \, \varphi \land \forall y \in b \, \exists x \in a \, \varphi$.

Throughout this paper, we will use \emph{denumerable} to mean ``in bijective correspondence with the set of natural numbers'' and \emph{finite} to mean ``in bijective correspondence with an initial segment of natural numbers''. A set which is either finite or denumerable, will be called \emph{countable}. In addition, we will call a set \emph{K-finite}, if it is the surjective image of an initial segment of the natural numbers. Observe that if a set has decidable equality, then it is finite if and only if it is K-finite.

In this paper we will also consider the following choice principles (countable choice and dependent choice):
\begin{displaymath}
\begin{array}{ll}
{\bf AC}_\omega & (\forall n \in \NN) (\exists x \in X) \varphi(n, x) \to (\exists f: \NN \to X) (\forall n \in \NN) \, \varphi(n, f(n)) \\ \\
{\bf DC}  & (\forall x \in X) \, (\exists y \in X) \, \varphi(x, y) \to \\ & (\forall x_0 \in X) \, (\exists f: \NN \to X) \, [ \, f(0) = x_0 \land (\forall n \in \NN) \, \varphi(f(n), f(n+1)) \, ]
\end{array}
\end{displaymath}
It is well-known that ${\bf DC}$ implies ${\bf AC}_\omega$, but not conversely (not even in {\bf ZF}). Any use of these additional axioms will be indicated explicitly.

\subsection{Inductive definitions in CZF}

\begin{defi}{closedsubclasses}
Let $S$ be a class. We will write Pow($S$) for the class of subsets of $S$. An \emph{inductive definition} is a subclass $\Phi$ of ${\rm Pow}(S) \times S$. One should think of the pairs $(X, a) \in \Phi$ as rules of the kind: if all elements in $X$ have a certain property, then so does $a$. Accordingly, a subclass $A$ of $S$ will be called \emph{$\Phi$-closed}, if
\[ X \subseteq A \Rightarrow a \in A \]
whenever $(X, a)$ is in $\Phi$.
\end{defi}

In {\bf CZF} one can prove that for any inductive definition $\Phi$ on a class $S$ and for any subclass $U$ of $S$ there is a least $\Phi$-closed subclass of $S$ containing $U$ (see \cite{aczelrathjen01}). We will denote this class by $I(\Phi, U)$. However, for the purposes of predicative locale theory one would like to have more:

\begin{theo}{setcompactness} {\rm (Set Compactness)}
If $S$ and $\Phi$ are sets, then there is a subset $B$ of ${\rm Pow}(S)$ such that for each set $U \subseteq S$ and each $a \in I(\Phi, U)$ there is a set $V \in B$ such that $V \subseteq U$ and $a \in I(\Phi, V)$.
\end{theo}

This result cannot be proved in {\bf CZF} proper, but it can be proved in extensions of {\bf CZF}. For example, this result becomes provable in {\bf CZF} extended with Aczel's regular extension axiom {\bf REA} \cite{aczelrathjen01} or in {\bf CZF} extended with the axioms {\bf WS} and {\bf AMC} \cite{bergmoerdijk10}. The latter extension is known to be stable under sheaves \cite{moerdijkpalmgren02, bergmoerdijk10}, while the former presumably is as well. Below, we will denote by {\bf CZF$^+$} any extension of {\bf CZF} which allows one to prove set compactness and which is stable under sheaves.

\section{Predicative locale theory}

In this section we have collected the definitions and results from predicative locale theory that we need in order to establish derived rules for {\bf CZF}. We have tried to keep our presentation self-contained, so that this section can actually be considered as a crash course on predicative locale theory or ``formal topology''. (In a predicative context, locales are usually called ``formal spaces'', hence the name. Some important references for formal topology are \cite{fourmangrayson82, coquandetal03, sambin03, aczel06} and, unless we indicate explicitly otherwise, the reader may find the results explained in this section in these sources.)

\subsection{Formal spaces}

\begin{defi}{formalspace}
A \emph{formal space} is a small site whose underlying category is a preorder. By a preorder, we mean a set $\mathbb{P}$ together with a small relation $\leq \, \subseteq \, \mathbb{P} \times \mathbb{P}$ which is both reflexive and transitive. If $a$ is an element of $\mathbb{P}$ then we will write $\downarrow a$ or $M_a$ for  $\{ p \in \mathbb{P} \, : \, p \leq a \}$, and if $\alpha$ is a subset of $\mathbb{P}$ then we will write $\downarrow \alpha = \{ p \in \mathbb{P} \, : \, (\exists a \in \alpha) \, p \leq a \}$. For the benefit of the reader, we repeat the axioms for a site from \cite{bergmoerdijk10b} for the special case of preorders.

Fix an element $a \in \mathbb{P}$. By a \emph{sieve} on $a$ we will mean a downwards closed \emph{subset} of $\downarrow a$. The set $M_a = \, \downarrow a$ will be called the \emph{maximal sieve} on $a$. In a predicative setting, the sieves on $a$ form in general only a class.

If $S$ is a sieve on $a$ and $b \leq a$, then we write $b^*S$ for the sieve
\[ b^*S = S \, \cap \downarrow b \]
on $b$. We will call this sieve \emph{the restriction of $S$ to $b$}.

A \emph{(Grothendieck) topology} Cov on $\mathbb{P}$ is given by assigning to every object $a \in \mathbb{P}$ a collection of sieves ${\rm Cov}(a)$ such that the following axioms are satisfied:
\begin{description}
\item[(Maximality)] The maximal sieve $M_a$ belongs to ${\rm Cov}(a)$.
\item[(Stability)] If $S$ belongs to ${\rm Cov}(a)$ and $b \leq a$, then $b^*S$ belongs to ${\rm Cov}(b)$.
\item[(Local character)] Suppose $S$ is a sieve on $a$. If $R \in {\rm Cov}(a)$ and all restrictions $b^*S$ to elements $b \in R$ belong to ${\rm Cov}(b)$, then $S \in {\rm Cov}(a)$.
\end{description}
A pair $(\mathbb{P}, {\rm Cov})$ consisting of a preorder $\mathbb{P}$ and a Grothendieck topology Cov on it is called a \emph{formal topology} or a \emph{formal space}. If a formal topology $(\mathbb{P}, {\rm Cov})$ has been fixed, the elements of $\mathbb{P}$ will be referred to as \emph{basic opens} or \emph{basis elements} and the sieves belonging to some ${\rm Cov}(a)$ will be referred to as the \emph{covering sieves}. If $S$ belongs to ${\rm Cov}(a)$ one says that $S$ is a \emph{sieve covering} $a$, or that $a$ \emph{is covered by} $S$.
\end{defi}

To develop a predicative theory of locales one needs to assume that the formal spaces one works with have a presentation, in the following sense. (Note that it was a standing assumption in \cite{bergmoerdijk10b} that sites had a presentation.)

\begin{defi}{setprformalsp}
A \emph{presentation} for a formal topology $(\mathbb{P}, {\rm Cov})$ is a function BCov assigning to every $a \in \mathbb{P}$ a \emph{small} collection of \emph{basic covering sieves} BCov($a$)  such that:
\[ S \in {\rm Cov}(a) \Leftrightarrow \exists R \in {\rm BCov}(a): R \subseteq S. \]
A formal topology which has a basis will be called \emph{presentable}.
\end{defi}

One of the ways in which presentable formal spaces behave better than general ones, is that, in {\bf CZF}, only presentable formal spaces give rise to categories of sheaves again modelling {\bf CZF} (see \reftheo{sheavessoundforCZF} and \refrema{presneeded} below). For this reason, it will be an important theme in this section to sort out whether the concrete formal spaces we work with can be proved to be presentable in {\bf CZF} or in extensions thereof.

\begin{rema}{curi}
Another property of presentable formal spaces, which is sometimes useful, is that for a given sieve $S$ the collection of $a \in \mathbb{P}$ which are covered by $S$ form a set, as it can be written as
\[ \{ a \in \mathbb{P} \, : \, (\exists R \in \mbox{BCov}(a)) \, (\forall r \in R) \, r \in S \}. \]
This property is often included in the definition of a formal space, but for our purposes this was not necessary. (We thank Giovanni Curi for bringing this issue to our attention; see also \cite{curi11}.)
\end{rema}

\subsection{Inductively generated formal topologies}

\begin{defi}{coveringsystem}
If $\mathbb{P}$ is a preorder, then a \emph{covering system} is a map $C$ assigning to every $a \in \mathbb{P}$ a small collection $C(a)$ of subsets of $\downarrow a$ such that the following covering axiom holds:
\begin{quote}
for every $\alpha \in C(p)$ and $q \leq p$, there is a $\beta \in C(q)$ such that \\ $\beta \subseteq q^* (\downarrow \alpha) = \{ r \leq q \, : \, (\exists a \in \alpha) \, r \leq a \}$.
\end{quote}
\end{defi}

Every covering system generates a formal space. Indeed, every covering system gives rise to an inductive definition $\Phi$ on $\mathbb{P}$, given by:
\[ \Phi = \{ (\alpha, a) \, : \, \alpha \in C(a) \}. \]
So we may define:
\[ S \in {\rm Cov}(a) \Leftrightarrow a \in I(\Phi, S). \]
Before we show that this is a Grothendieck topology, we first note:

\begin{lemm}{sievesgensieves}
If $S$ is a downwards closed subclass of $\downarrow a$, then so is $I(\Phi, S)$. Also, $x \in I(\Phi, S)$ iff $x \in I(\Phi, x^*S)$.
\end{lemm}
\begin{proof}
The class $I(\Phi, S)$ is inductively generated by the rules:
\begin{displaymath}
\begin{array}{cc}
\frac{ r \in S}{r \in I(\Phi, S)} & \frac{\alpha \subseteq I(\Phi, S) \quad \alpha \in C(r)}{r \in I(\Phi, S)}
\end{array}
\end{displaymath}
Both statements are now proved by an induction argument, using the covering axiom.
\end{proof}

\begin{theo}{coveringsystem}
Every covering system generates a formal topology. More precisely, for every covering system $C$ there is a smallest Grothendieck topology ${\rm Cov}$ such that
\[ \alpha \in C(a) \Longrightarrow \downarrow \alpha \in {\rm Cov}(a). \]
In ${\bf CZF^+}$ one can show that this formal topology has a presentation.
\end{theo}
\begin{proof}
Note that the Cov relation is inductively generated by:
\begin{displaymath}
\begin{array}{cc}
\frac{ a \in S}{S \in {\rm Cov}(a)} & \frac{\alpha \in C(a) \quad (\forall x \in \alpha) \, x^*S \in {\rm Cov}(x)}{S \in {\rm Cov}(a)}
\end{array}
\end{displaymath}
Maximality is therefore immediate, while stability and local character can be established using straightforward induction arguments. Therefore Cov is indeed a topology. The other statements of the theorem are clear.
\end{proof}

\begin{theo}{indoncovers} {\rm (Induction on covers)} Let $(\mathbb{P}, {\rm Cov})$ be a formal space, whose topology ${\rm Cov}$ is inductively generated by a covering system $C$, as in the previous theorem. Suppose $P(x)$ is a property of basis elements $x \in \mathbb{P}$, such that
\[ \forall \alpha \in C(x) \big( \, ((\forall y \in \alpha) \, P(y)) \to P(x) \big), \]
and suppose $S$ is a cover of an element $a \in \mathbb{P}$ such that $P(y)$ holds for all $y \in S$. Then $P(a)$ holds.
\end{theo}
\begin{proof}
Suppose $P$ has the property in the hypothesis of the theorem. Define:
\[ S \in {\rm Cov}^*(p) \Leftrightarrow (\forall q \leq p) \, \big( \, ((\forall r \in q^*S) \, P(r)) \to P(q) \, \big). \]
Then one checks that ${\rm Cov}^*$ is a topology extending $C$. So by \reftheo{coveringsystem} we have $S \in {\rm Cov}(a) \subseteq {\rm Cov}^*(a)$, from which the desired result follows.
\end{proof}

\subsection{Formal Baire and Cantor space}

We will write $X^{\lt \NN}$ for the set of finite sequences of elements from $X$. Elements of $X^{\lt \NN}$ will usually be denoted by the letters $u, v, w, \ldots$. Also, we will write $u \leq v$ if $v$ is an initial segment of $u$, $|v|$ for the length of $v$ and $u * v$ for the concatenation of sequences $u$ and $v$. If $u \in X^{\lt \NN}$ and $q \geq |u|$ is a natural number, then we define $u[q]$ by:
\[ u[q] = \{ v \in X^{\lt \NN} \, : \, |v| = q \mbox{ and } v \leq u \}. \]

The basis elements of formal Cantor space $\mathbf{C}$ are finite sequences $u \in 2^{\lt \NN}$ (with $2 = \{ 0, 1 \}$), ordered by saying that $u \leq v$, whenever $v$ is an initial segment of $u$. Furthermore, we put
\[ S \in {\rm Cov}(u) \Leftrightarrow (\exists q \geq |u|) \, u[q]  \subseteq  S \]
and ${\rm BCov}(u) = \{ u[q] \, : \, q \geq |u| \}$. Note that this will make formal Cantor space compact \emph{by definition} (where a formal space is compact, if for every cover $S$ of $u$ there is a K-finite subset $\alpha$ of $S$ such that $\downarrow \alpha \in {\rm Cov}(u)$).

\begin{prop}{CSassite}
Formal Cantor space is a presentable formal space.
\end{prop}
\begin{proof}
We leave maximality and stability to the reader and only check local character. Suppose $S$ is a sieve on $u$ for which a sieve $R \in {\rm Cov}(u)$ can be found such that for all $v \in R$ the sieve $v^*S = (\downarrow v) \cap S$ belongs to ${\rm Cov}(v)$. Since $R \in {\rm Cov}(u)$ there is $q \geq |u|$ such that $u[q] \subseteq R$. Therefore we have for any $v \in u[q]$ that $(\downarrow v) \cap S$ covers $v$ and hence that there is a $r \geq q$ such that $v[r] \subseteq S$. Since the set $u[q]$ is finite, the elements $r$ can be chosen as a function $v$. For $p = \max\{ r_v \, : \, v \in u[q] \}$, it holds that
\[ u[p] = \bigcup_{v \in u[q]} v[p] \subseteq S, \]
as desired.
\end{proof}

Formal Baire space $\mathbf{B}$ is an example of an inductively defined space. The underlying poset has as elements finite sequences $u \in \NN^{\lt \NN}$, ordered as for Cantor space above. The Grothendieck topology is inductively generated by:
\[ C(u) = \big\{ \{ u \, *  \, \langle n \rangle \, : \, n \in \N \} \big \}, \]
and therefore we have the following induction principle:
\begin{coro}{formalBIM} {\rm (Bar Induction for formal Baire space)}
Suppose $P(x)$ is a property of finite sequences $u \in \N^{\lt \NN}$, such that
\[ \big( (\forall n \in \N) \, P(u * \langle n \rangle ) \big) \to P(u), \]
and suppose that $S$ is a cover of $v$ in formal Baire space such that $P(x)$ holds for all $x \in S$. Then $P(v)$ holds.
\end{coro}
Note that this means that Bar Induction for formal Baire space is \emph{provable}. 

As a special case of \reftheo{coveringsystem} we have:
\begin{coro}{formalBSpresentableinCZF+}
In {\rm {\bf CZF$^+$}} one can show that formal Baire space is presentable.
\end{coro}
In contrast, we observe that formal Baire space cannot be shown to be presentable in {\bf CZF} proper (a proof can be found in Appendix A).

\subsection{Points of a formal space}

The characteristic feature of formal topology is that one takes the notion of basic open as primitive and the notion of a point as derived. In fact, the notion of a point is defined as follows:
\begin{defi}{point}
A \emph{point} of a formal space $(\mathbb{P},  \mbox{Cov})$ is an inhabited subset $\alpha \subseteq \mathbb{P}$ such that
\begin{enumerate}
\item[(1)] $\alpha$ is upwards closed,
\item[(2)] $\alpha$ is downwards directed,
\item[(3)] if $S \in \mbox{Cov}(a)$ and $a \in \alpha$, then $S \cap \alpha$ is inhabited.
\end{enumerate}
We say that a point $\alpha$ belongs to (or is contained in) a basic open $p \in \mathbb{P}$ if $p \in \alpha$, and we will write ${\rm ext}(p)$ for the class of points of the basic open $p$.
\end{defi}

If $(\mathbb{P}, {\rm Cov})$ is a formal space whose points form a set, one can define a new formal space ${\rm pt}(\mathbb{P}, {\rm Cov}) = (\mathbb{P}_{pt}, {\rm Cov}_{pt})$, whose set of basic opens $\mathbb{P}_{pt}$ is again $\mathbb{P}$, but now ordered by:
\[ p \leq_{pt} q \Leftrightarrow {\rm ext}(p) \subseteq {\rm ext}(q), \]
while the topology is defined by:
\[ S \in {\rm Cov}_{pt}(a) \Leftrightarrow {\rm ext}(a) \subseteq \bigcup_{p \in S} {\rm ext}(p). \]
The space ${\rm pt}(\mathbb{P}, {\rm Cov})$ will be called the \emph{space of points} of the formal space $(\mathbb{P}, {\rm Cov})$. It follows immediately from the definition of a point that
\begin{eqnarray*}
p \leq q & \Rightarrow & p \leq_{pt} q, \\
S \in {\rm Cov}(a) & \Rightarrow & S \in {\rm Cov}_{pt}(a).
\end{eqnarray*}
The other directions of these implications do not hold, in general. Indeed, if they do, one says that the formal space \emph{has enough points}. It turns out that one can quite easily construct formal spaces that do not have enough points (even in a classical metatheory).

Note that points in formal Cantor space are really functions $\alpha: \N \to \{ 0, 1 \}$ and points in formal Baire space are functions $\alpha: \N \to \N$. In fact, their spaces of points are (isomorphic to) ``true'' Cantor space and ``true'' Baire space, respectively. When talking about such points $\alpha$ we will use $\alpha$ both to denote a subset as in \refdefi{point} and a function on $\NN$. In particular, the equivalent notations
\[ u \in \alpha \Leftrightarrow \alpha \leq u \]
both indicate that $u$ is an initial segment of $\alpha$.

The following two results were already mentioned in the introduction and are well-known in the impredicative settings of topos theory or intuitionistic set theory {\bf IZF}. Here we wish to emphasise that they hold in {\bf CZF} as well.

\begin{prop}{pointsincantorspace}
The following statements are equivalent:
\begin{enumerate}
\item[(1)] Formal Cantor space has enough points.
\item[(2)] Cantor space is compact.
\item[(3)] The Fan Theorem: If $S$ is a downwards closed subset of $2^{\lt \N}$ and
\[ (\forall \alpha \in 2^{\N}) \, (\exists u \in \alpha) \, u \in S, \]
then there is a $q \in \mathbb{N}$ such that $< \, >[q] \subseteq S$.
\end{enumerate}
\end{prop}
\begin{proof}
The equivalence of (2) and (3) holds by definition of compactness and the equivalence of (1) and (3) by the definition of having enough points.
\end{proof}

\begin{prop}{pointsinbairespace}
The following statements are equivalent:
\begin{enumerate}
\item[(1)] Formal Baire space has enough points.
\item[(2)] Monotone Bar Induction: If $S$ is a downwards closed subset of $\N^{\lt \N}$ and
\[ (\forall \alpha \in \N^{\N}) \, (\exists u \in \alpha) \, u \in S \]
and
\[ (\forall u \in \N^{\lt \N}) \big( ((\forall n \in \N) \, u * <n> \in S) \to u \in S\,  \big) \]
hold, then $< \, > \in S$.
\end{enumerate}
\end{prop}
\begin{proof}
(1) $\Rightarrow$ (2): If $S$ is downwards closed and satisfies $(\forall \alpha \in \N^{\N}) \, (\exists u \in \alpha) \, u \in S$ (i.e., $S$ is a bar), then $S \in {\rm Cov}_{pt}(<>)$, by definition. From the hypothesis that formal formal Baire space has enough points it then follows that $S \in {\rm Cov}(<>)$. Hence, if $S$ is inductive as well (i.e., satisfies $(\forall u \in \N^{\lt \N}) \big( ((\forall n \in \N) \, u * <n> \in S) \to u \in S\,  \big)$), then one may apply Monotone Bar Induction for formal Baire space (\refcoro{formalBIM}) to deduce that $<> \in S$.

(2) $\Rightarrow$ (1): Assume that Monotone Bar Induction holds and suppose that $S \in {\rm Cov}_{pt}(< \, >)$ is arbitrary. We have to show that $S \in {\rm Cov}(< \, >)$. By definition, this means that we have to show that $< \, > \in \overline{S}$, where $\overline{S}$ is inductively defined by the rules:
\begin{displaymath}
\begin{array}{cc}
\frac{ a \in S}{a \in \overline{S}} & \frac{(\forall n \in \N) \, u * <n> \in \overline{S}}{u \in \overline{S}}
\end{array}
\end{displaymath}
(see the construction just before \reflemm{sievesgensieves}). However, since $\overline{S}$ is downwards closed (by \reflemm{sievesgensieves}), a bar (because $S$ is a bar and $S \subseteq \overline{S}$) and inductive (by construction), we may apply Monotone Bar Induction to $\overline{S}$ to deduce that $< \, > \in \overline{S}$, as desired.
\end{proof}

\subsection{Morphisms of formal spaces}

Points are really a special case of morphisms of formal spaces. Here we will assume that the formal spaces we consider are presentable or at least satisfy the additional condition that the collection of basis elements covered by a fixed sieve form a set.

\begin{defi}{morphismofformalspaces}
A \emph{continuous map} or a \emph{morphism of formal spaces} $F: (\mathbb{P}, {\rm Cov}) \to (\mathbb{Q}, {\rm Cov'})$ is a subset $F \subseteq P \times Q$ such that:
\begin{enumerate}
\item[(1)] If $F(p, q)$, $p' \leq p$ and $q \leq q'$, then $F(p', q')$.
\item[(2)] For every $p \in \mathbb{P}$ there is a cover $S \in {\rm Cov}(p)$ such that each $p' \in S$ is related via $F$ to some element $q' \in \mathbb{Q}$.
\item[(3)] For every $q, q' \in \mathbb{Q}$ and element $p \in \mathbb{P}$ such that $F(p, q)$ and $F(p, q')$, there is a cover $S \in {\rm Cov}(p)$ such that every $p' \in S$ is related via $F$ to an element which is smaller than or equal to both $q$ and $q'$.
\item[(4)] Whenever $F(p, q)$ and $T$ covers $q$, there is a sieve $S$ covering $p$, such that every $p' \in S$ is related via $F$ to some $q' \in T$.
\item[(5)] For every $q \in \mathbb{Q}$, the set $\{ p \, : \, F(p, q) \}$ is  closed under the covering relation.
\end{enumerate}
In condition (5) we say that a sieve $S$ is \emph{closed under covering relation} (or simply \emph{closed}), if
\[ R \in {\rm Cov}(a), R \subseteq S \Longrightarrow a \in S. \]
\end{defi}

To help the reader to make sense of this definition, it might be good to recall  some facts from locale theory (see \cite{johnstone82}). A \emph{locale} is a partially ordered class $\mathcal{A}$ which has finite meets and small joins, with the small joins distributing over the finite meets. In addition, a morphism of locales $\mathcal{A} \to \mathcal{B}$ is a map $\mathcal{B} \to \mathcal{A}$ preserving finite meets and small joins.

Every formal space $(\mathbb{P}, {\rm Cov})$ determines a locale ${\rm Idl}(\mathbb{P}, {\rm Cov})$, whose elements are the closed sieves on $\mathbb{P}$, ordered by inclusion. Moreover, every morphism of locales $\varphi: {\rm Idl}(\mathbb{P}, {\rm Cov}) \to {\rm Idl}(\mathbb{Q}, {\rm Cov'})$ determines a relation $F \subseteq P \times Q$ by $p \in \varphi(\overline{q})$, with $\overline{q}$ being the least closed sieve containing $q$. The reader should verify that this relation $F$ has the properties of a map of formal spaces and that every such $F$ determines a unique morphism of locales $\varphi: {\rm Idl}(\mathbb{P}, {\rm Cov}) \to {\rm Idl}(\mathbb{Q}, {\rm Cov'})$.

Together with the continuous maps the class of formal spaces organises itself into a superlarge category, with composition given by composition of relations and identity $I: (\mathbb{P}, {\rm Cov}) \to (\mathbb{P}, {\rm Cov})$ by
\[ I(p, q) \Longleftrightarrow (\exists S \in {\rm Cov}(p)) \, (\forall r \in S) \, r \leq q. \]
(if the formal space is \emph{subcanonical} ($\overline{p} = \downarrow p$ for all $p \in P$),  this simplifies to $I(p, q)$ iff $p \leq q$). Note that in a predicative metatheory, this category cannot be expected to be locally small.

A point of a formal space $(\mathbb{P}, {\rm Cov})$ is really the same thing as a map $1 \to (\mathbb{P}, {\rm Cov})$, where 1 is the one-point space $(\{ * \}, {\rm Cov}')$ with ${\rm Cov'}(*) = \big\{ \{ * \} \big \}$. Indeed, if $F: 1 \to (\mathbb{P}, {\rm Cov})$ is a map, then $\alpha = \{ p \in \mathbb{P} \, : \, F(*, p) \}$ is a point, and, conversely, if $\alpha$ is a point, then
\[ F(*, p) \Leftrightarrow p \in \alpha \]
defines a map. Moreover, these operations are clearly mutually inverse. This implies that any continuous map $F: (\mathbb{P}, {\rm Cov}) \to (\mathbb{Q}, {\rm Cov'})$ induces a function ${\rm pt}(F): {\rm pt}(\mathbb{P}, {\rm Cov}) \to {\rm pt}(\mathbb{Q}, {\rm Cov'})$ (by postcomposition). Since this map is continuous, ${\rm pt}$ defines an endofunctor on the category of those formal spaces on which ${\rm pt}$ is well-defined. 

In addition, we have for any formal space $(\mathbb{P}, {\rm Cov})$ on which ${\rm pt}$ is well-defined a continuous map $F: {\rm pt}(\mathbb{P}, {\rm Cov}) \to (\mathbb{P}, {\rm Cov})$ given by $F(p, q)$ iff ${\rm ext}(p) \subseteq {\rm ext}(q)$. This map $F$ is an isomorphism precisely when $(\mathbb{P}, {\rm Cov})$ has enough points. (In fact, $F$ is the component at $(\mathbb{P}, {\rm Cov})$ of a natural transformation ${\rm pt} \Rightarrow \id$.)

\subsection{The double of a formal space}

Although the Fan Theorem and Monotone Bar Induction are not provable in {\bf CZF}, we will show below that they do hold as derived rules. For that purpose, we use a construction on formal spaces, which is due to Joyal and which we have dubbed the ``double construction''.\footnote{See \cite[Section 15.4]{troelstravandalen88b}.  This construction is known in the impredicative case for locales, but here we wish to emphasise that it works in a predicative setting for formal spaces as well.} This construction will enable us to relate the sheaf semantics over a formal space to external truth (see part 3 of \reflemm{lemmaonforcing}). The best way to explain it is to consider the analogous construction for ordinary topological spaces first.

Starting from a topological space $X$, the double construction takes two disjoint copies of $X$, so that every subset of it can be considered as a pair $(U, V)$ of subsets of $X$. Such a pair will be open, if $U$ is open in $X$ and $U \subseteq V$. Note that we do not require $V$ to be open in $X$: $V$ can be an arbitrary subset of $X$.
Therefore the double construction can be seen as a kind of mapping cylinder with Sierpi\'nski space replacing the unit interval: the ordinary mapping cylinder of a map $f: Y \to X$ is obtained by taking the space $[0,1] \times Y + X$ and then identifying points $(0, y)$ with $f(y)$ (for all $y \in Y$). The double of a space $X$ is obtained from this construction by replacing the unit interval $[0,1]$ by Sierpi\'nski space and considering the canonical map $X_{discr} \to X$.

The construction for formal spaces is now as follows: suppose $({\mathbb U}, {\rm Cov})$ is a formal space whose points form a set $Q$. The set of basic opens of ${\cal D}(\mathbb{U}, {\rm Cov}) = (\mathbb{U}_D, {\rm Cov}_D)$ is
\[ \big\{ D(u) \, : \, u \in \mathbb{U} \big\} + \big\{ \{ q \} \, : q \in Q \big\}. \]
Here both $D(u)$ and $\{ q \}$ are formal symbols for a basic open, representing the pairs $(u, u)$ and $(\emptyset, \{ q \})$ in the topological case. The preorder on $\mathbb{U}_D$ is generated by:
\begin{displaymath}
\begin{array}{cl}
D(v) \leq D(u) & \mbox{if } v \leq u \mbox{ in } \mathbb{U}, \\
\{ q \} \leq D(v) & \mbox{if } v \in q, \\
\{ p \} \leq \{ q \} & \mbox{if } p = q.
\end{array}
\end{displaymath}
In addition, the covering relation is given by
\begin{eqnarray*}
{\rm Cov}_D(D(u)) & = & \big\{ \{ D(v) \, : \, v \in S \} \cup \{ \{q \}  \, : \, v \in q, v \in S \} \, : \, S \in {\rm Cov}(u) \big\}, \\
{\rm Cov}_D( \{ q \}) & = & \big\{ \{ q \} \big\}.
\end{eqnarray*}
\begin{prop}{doublelemma}
${\cal D}(\mathbb{U}, {\rm Cov})$ as defined above is a formal space, which  is presentable, whenever $(\mathbb{U}, {\rm Cov})$ is.
\end{prop}
\begin{proof}
This routine verification we leave to the reader. Note that if BCov is a presentation for  the covering relation Cov, then 
\begin{eqnarray*}
{\rm BCov}_D(D(u)) & = & \big\{ \{ D(v) \, : \, v \in S \} \, : \, S \in {\rm Cov}(u) \big\}, \\
{\rm BCov}_D( \{ q \}) & = & \big\{ \{ q \} \big\}
\end{eqnarray*}
is a presentation for ${\rm Cov}_D$.
\end{proof}

The formal space ${\cal D}(\mathbb{U}, {\rm Cov})$ comes equipped with three continuous maps. First of all, there is a closed map $\mu: (\mathbb{U}, {\rm Cov}) \to {\cal D}(\mathbb{U}, {\rm Cov})$ given by $\mu(u, p)$ iff $p = D(v)$ for some $v \in \mathbb{U}$ with $I(u, v)$. In addition, there is a map $\pi: {\cal D}(\mathbb{U}, {\rm Cov}) \to (\mathbb{U}, {\rm Cov})$ given by $\pi(p, u)$ iff there is a $v \in \mathbb{U}$ with $u = D(v)$ and $I(v, u)$. Note that $\pi \circ \mu = \id$. And, finally, there is an open map of the form $\nu: (\mathbb{U}, {\rm Cov})_{discr} \to {\cal D}(\mathbb{U}, {\rm Cov})$. The domain of this map $(\mathbb{U}, {\rm Cov})_{discr}$ is the formal space whose basic opens are singletons $\{ q \}$ (with the discrete ordering) and whose only covering sieves are the maximal ones. The map $\nu$ is then given by $\nu(\{ q \}, u)$ iff $u = \{ q \}$. We depict these maps in the following diagram:
\diag{ (\mathbb{U}, {\rm Cov}) \ar[rr]^{\mu} & & {\cal D}(\mathbb{U}, {\rm Cov}) \ar[d]^{\pi} & & (\mathbb{U}, {\rm Cov})_{discr} \ar[ll]_{\nu} \\
& & (\mathbb{U}, {\rm Cov}). }

\section{Sheaf models}

In \cite{gambino02} and \cite{bergmoerdijk10b} it is shown how sheaves over a presentable formal space give rise to a model of {\bf CZF}. Moreover, since this fact is provable within {\bf CZF} itself, sheaf models can be used to establish proof-theoretic facts about {\bf CZF}, such as derived rules. We will exploit this fact to prove Derived Fan and Bar Induction rules for (extensions of) {\bf CZF}.

\subsection{Basic properties of sheaf semantics}

We recapitulate the most important facts about sheaf models below. We hope this allows the reader who is not familiar with sheaf models to gain the necessary informal understanding to make sense of the proofs in this section. The reader who wants to know more or wishes to see some proofs, should consult \cite{gambino02} and \cite{bergmoerdijk10b}.

A \emph{presheaf} $X$ over a preorder $\mathbb{P}$ is a functor $X: \mathbb{P}^{op} \to \Sets$. This means that $X$ is given by a family of sets $X(p)$, indexed by elements $p \in \mathbb{P}$, and a family of restriction operations $- \upharpoonright q: X(p) \to X(q)$ for $q \leq p$, satisfying:
\begin{enumerate}
\item $- \upharpoonright p: X(p) \to X(p)$ is the identity,
\item for every $x \in X(p)$ and $r \leq q \leq p$, $(x \upharpoonright q) \upharpoonright r = x \upharpoonright r$.
\end{enumerate}
Given a topology Cov on $\mathbb{P}$, a presheaf $X$ will be called a \emph{sheaf}, if it satisfies the following condition:
\begin{quote}
For any given sieve $S \in {\rm Cov}(p)$ and family $\{ x_q \in X(q) \, : \, q \in S \}$, which is compatible, meaning that $(x_q) \upharpoonright r = x_r$ for every $r \leq q \in S$, there is a unique $x \in X(p)$ (the ``amalgamation'' of the compatible family) such that $x \upharpoonright q = x_q$ for all $q \in S$.
\end{quote}

\begin{lemm}{shaxiomredtocovsystem}
If a formal space $(\mathbb{P}, {\rm Cov})$ is generated by a covering system $C$, then it suffices to check the sheaf axiom for those families which belong to the covering system.
\end{lemm}
\begin{proof}
Suppose $X$ is a presheaf satisfying the sheaf axiom with respect to the covering system $C$, in the following sense:
\begin{quote}
For any given element $\alpha \in C(a)$ and family $\{ x_q \in X(q) \, : \, q \in \alpha \}$, which is compatible, meaning that for all $r \leq p, q$ with $p, q \in \alpha$ we have $(x_p) \upharpoonright r = (x_q) \upharpoonright r$, there exists a unique $x \in X(a)$ such that $x \upharpoonright q = x_q$ for all $q \in \alpha$.
\end{quote}
Define ${\rm Cov}^*$ by:
\begin{eqnarray*}
S \in {\rm Cov}^*(a) & \Leftrightarrow & \mbox{if } b \leq a \mbox{ and } \{ x_q \in X(q) \, : \, q \in b^*S \} \mbox{ is a compatible family,} \\
& & \mbox{then it can be amalgamated to a unique } x \in X(b).
\end{eqnarray*}
${\rm Cov}^*$ is a Grothendieck topology, which, by assumption, satisfies
\[ \alpha \in C(a) \Longrightarrow \downarrow \alpha \in {\rm Cov}^*(a). \]
Therefore ${\rm Cov} \subseteq {\rm Cov}^*$, which implies that $X$ is a sheaf with respect to the Grothendieck topology ${\rm Cov}$.
\end{proof}

A morphism of presheaves $F: X \to Y$ is a natural transformation, meaning that it consists of functions $\{ F_p: X(p) \to Y(p) \, : \, p \in \mathbb{P} \}$ such that for all $q \leq p$ we have a commuting square:
\diag{ X(p) \ar[r]^{F_p} \ar[d]_{- \upharpoonright q} & Y(p) \ar[d]^{- \upharpoonright q} \\
X(q) \ar[r]_{F_q} & Y(q). }
The category of sheaves is a full subcategory of the category of presheaves, so every natural transformation $F: X \to Y$ between sheaves $X$ and $Y$ is regarded as a morphism of sheaves.

The category of sheaves is a Heyting category and therefore has an ``internal logic''. This internal logic can be seen as a a generalisation of forcing, in that truth in the model can be explained using a binary relation between elements $p \in \mathbb{P}$ (the ``conditions'' in forcing speak) and first-order formulas. This forcing relation is inductively defined as follows:
\begin{eqnarray*}
p \forces \varphi \land \psi & \Leftrightarrow & p \forces \varphi \mbox{ and } p \forces \psi \\
p \forces \varphi \lor \psi & \Leftrightarrow & \{ q \leq p \, : \, q \forces \varphi \mbox{ or } q  \forces \psi  \} \in {\rm Cov}(p) \\
p \forces \varphi \to \psi & \Leftrightarrow & (\forall q \leq p) \, q \forces \varphi \Rightarrow q \forces \psi \\
p \forces \bot & \Leftrightarrow & \emptyset \in {\rm Cov}(p)  \\
p \forces (\exists x : X) \, \varphi(x) & \Leftrightarrow & \{ q \leq p \, : \, (\exists x \in X(q)) \, q \forces \varphi(x) \} \in {\rm Cov}(p) \\
p \forces (\forall x : X) \, \varphi(x) & \Leftrightarrow & (\forall q \leq p) \, (\forall x \in X(q)) \, q \forces \varphi(x)
\end{eqnarray*}

\begin{lemm}{lemmaonforcing}
Sheaf semantics has the following properties:
\begin{enumerate}
\item (Monotonicity) If $p \forces \varphi$ and $q \leq p$, then $q \forces \varphi$.
\item (Local character) If $S$ covers $p$ and $q \forces \varphi$ for all $q \in S$, then $p \forces \varphi$.
\item If $p$ is minimal (so $q \leq p$ implies $q = p$) and ${\rm Cov}(p) = \big\{ \{ p \} \big \}$, then forcing at $p$ coincides with truth, i.e., we have $\varphi$ iff $p \forces \varphi$.
\end{enumerate}
\end{lemm}
\begin{proof}
By induction on the structure of $\varphi$.
\end{proof}

Note, in connection with Section 3.6, that every element of the form $\{ q \}$ in the double forms a minimal element to which the hypothesis of part 3 of \reflemm{lemmaonforcing} applies.

Using this forcing relation, one defines truth in the model as being forced by every condition $p \in \mathbb{P}$. If $\mathbb{P}$ has a top element 1, this coincides with being forced at this element (by monotonicity).

One way to see sheaf semantics is as a generalisation of forcing for classical set theory, which one retrieves by putting:
\[ S \in {\rm Cov}(p) \Leftrightarrow S \mbox{ is dense below } p. \]
Forcing for this specific forcing relation validates classical logic, but in general sheaf semantics will only validate intuitionistic logic.

Sheaf semantics as described above is a way of interpreting first-order theories in a category of sheaves over $(\mathbb{P}, {\rm Cov})$. To obtain a semantics for the language of set theory, one uses the machinery of algebraic set theory and proceeds as follows (see \cite{moerdijkpalmgren02,bergmoerdijk08,bergmoerdijk09}). Let $\pi: E \to U$ be the universal small map in the category of sheaves and let $V = W(\pi)/\sim$ be the extensional (Mostowski) collapse of the W-type $W = W(\pi)$. Like any W-type, $W$ comes equipped with a relation $M$ generated by
\[ t(e) \, M \, {\rm sup}_u(t) \]
for any $e \in E_u$ and $t: E_u \to W$. This relation $M$ descends to a well-defined relation on $V$, which interprets the membership symbol in the language of set theory and will be denoted by $\epsilon$. For the resulting model $(V, \epsilon)$ we have:
\begin{theo}{sheavessoundforCZF}
If $(\mathbb{P}, {\rm Cov})$ is a presentable formal space, then sheaf semantics over $(\mathbb{P},  {\rm Cov})$ is sound for {\bf CZF}, as it is for   {\bf CZF} extended with small W-types {\bf WS} and the axiom of multiple choice {\bf AMC}. Moreover, the former is provable within {\bf CZF}, while the latter is provable in {\bf CZF + WS + AMC}.
\end{theo}
\begin{proof}
This is proved in \cite{bergmoerdijk10b, moerdijkpalmgren02} for the general case of sheaves over a site. For the specific case of sheaves on a formal space and {\bf CZF} alone, this was proved earlier by Gambino in terms of Heyting-valued models  \cite{gambino02,gambino06}.
\end{proof}

\begin{rema}{presneeded}
The requirement that $(\mathbb{P}, {\rm Cov})$ has a presentation is essential: the theorem is false without it (see \cite{gambino06}). Therefore we will assume from now on that $(\mathbb{P}, {\rm Cov})$ is presentable.
\end{rema}

For the proofs below we need to compute various objects related to Cantor space and Baire space in different categories of sheaves. We will discuss the construction of $\N$ in sheaves in some detail: this will hopefully give the reader sufficiently many hints to see why the formulas we give for the others are correct. 

To compute $\NN$ in sheaves, one first computes $\NN$ in presheaves, where it is pointwise constant $\NN$. The corresponding object in sheaves is obtained by sheafifying this object, which means by twice applying the plus-construction (the standard treatment as in \cite{maclanemoerdijk92} can also be followed in {\bf CZF}). In case every covering sieve is inhabited, the presheaf $\N$ is already separated, so then it suffices to apply the plus-construction only once. In that case, we obtain:
\begin{eqnarray*}
\NN(p) & = & \{ (S, \varphi) \, : \, S \in {\rm Cov}(p), \varphi: S \to \NN \mbox{ compatible} \} / \sim,
\end{eqnarray*}
with $(S, \varphi) \sim (T, \psi)$, if there is an $R \in {\rm Cov}(p)$ with $R \subseteq S \cap T$ and $\varphi(r) = \psi(r)$ for all $r \in R$, and $(S, \varphi) \upharpoonright q = (q^*S, \varphi \restriction q^*S)$.

\begin{rema}{naturalnumbersandcontinuousmaps}
If $\mathbb{P}$ has a top element 1 (as often is the case), then elements of $\N(1)$ correspond to continuous functions
\[ (\mathbb{P}, {\rm Cov}) \to \N_{discr}. \]
\end{rema}

\begin{rema}{pureelements}
Borrowing terminology from Boolean-valued models \cite{bell85}, we could call elements of $\NN(p)$ of the form $(M_p, \varphi)$ \emph{pure} and others \emph{mixed} (recall that $M_p = \downarrow p$ is the maximal sieve on $p$). As one sees from the description of $\NN$ in sheaves, the pure elements lie dense in this object, meaning that for every $x \in \NN(p)$,
\[ \{ q \leq p \, : \, x \upharpoonright q \mbox{ is pure} \} \in {\rm Cov}(p). \]
This, together with the local character of sheaf semantics, has the useful consequence that in the clauses for the quantifiers
\begin{eqnarray*}
p \forces (\exists x \in \NN) \, \varphi(x) & \Leftrightarrow & \{ q \leq p \, : \, (\exists x \in \NN(q)) \, q \forces \varphi(x) \} \in {\rm Cov}(p) \\
p \forces (\forall x \in \NN) \, \varphi(x) & \Leftrightarrow & (\forall q \leq p) \, (\forall x \in \NN(q)) \, q \forces \varphi(x)
\end{eqnarray*}
one may restrict one's attention to those $x \in \NN(q)$ that are pure.
\end{rema}

We also have the following useful formulas:
\begin{eqnarray*}
2(p) & = & \{ (S, \varphi) \, : \, S \in {\rm Cov}(p), \varphi: S \to \{ 0, 1 \} \mbox{ compatible} \} / \sim,\\
2^{\lt \NN}(p) & = &  2(p)^{\lt \NN}, \\
2^\NN(p) & = & 2(p)^\NN, \\
\NN^{\lt \NN}(p) & = &  \NN(p)^{\lt \NN}, \\
\NN^\NN(p) & = & \NN(p)^\NN.
\end{eqnarray*}
All these objects come equipped with the obvious equivalence relations and restriction operations. We will not show the correctness of these formulas, which relies heavily on the following fact:
\begin{prop}{sheavesexpideal} {\rm \cite[Proposition III.1, p.~136]{maclanemoerdijk92}}
The sheaves form an exponential ideal in the category of presheaves, so if $X$ is a sheaf and $Y$ is a presheaf, then $X^Y$ (as computed in presheaves) is a sheaf.
\end{prop}
From these formulas one sees that, if $\mathbb{P}$ has a top element 1, then $2^\N(1)$ can be identified with the set of continuous functions $(\mathbb{P}, {\rm Cov}) \to \mathbf{C}$ to formal Cantor space and $\N^\N(1)$ with the set of continuous functions $(\mathbb{P}, {\rm Cov}) \to \mathbf{B}$ to formal Baire space. Also, in $2^{\lt \N}$ and $\N^{\lt \N}$ the ``pure'' elements are again dense. (But this is not true for $2^\N$ and $\N^\N$, in general.)

\subsection{Choice principles}

For our purposes it will be convenient to introduce the following \emph{ad hoc} terminology.

\begin{defi}{CCspace}
A formal space $(\mathbb{P}, {\rm Cov})$ will be called a \emph{CC-space}, if every cover has a countable, disjoint refinement. This means that for every $S \in {\rm Cov}(p)$, there is a countable $\alpha \subseteq S$ such that  $\downarrow \alpha \in {\rm Cov}(p)$ and for all $p, q \in \alpha$, either $p = q$ or $\downarrow p \, \cap \downarrow q = \emptyset$.
\end{defi}

\begin{exam}{examoCCspaces}
Formal Cantor space is a CC-space and if ${\bf AC}_\omega$ holds, then so is formal Baire space (see \refprop{BSassiteCAC}). Also, doubles of CC-spaces are again CC.
\end{exam}

Our main reason for introducing the notion of a CC-space is the following proposition, which is folklore (see, for instance, \cite{grayson83}):
\begin{prop}{preservationofchoice}
Suppose $(\mathbb{P}, {\rm Cov})$ is a presentable formal space which is CC. If ${\bf DC}$ or ${\bf AC}_\omega$ holds in the metatheory, then the same choice principle holds in ${\rm Sh}(\mathbb{P}, {\rm Cov})$. Moreover, this fact is provable in {\bf CZF}.
\end{prop}
\begin{proof}
We check this for ${\bf AC_\omega}$, the argument for ${\bf DC}$ being very similar. So suppose $X$ is some sheaf and
\[ p \forces (\forall n \in \NN) (\exists x \in X) \, \varphi(n, x). \]
Using that the pure elements in $\NN$ are dense (\refrema{pureelements}), this means that for every $n \in \NN$ there is a cover $S \in {\rm Cov}(p)$ such that for all $q \in S$ there is an $x \in X(q)$ such that
\[ q \forces \varphi(n, x). \]
Because the space is assumed to be CC we have $S = \downarrow \alpha$ for a set $\alpha$ which is countable and disjoint. Furthermore, since ${\bf AC}_\omega$ holds, the $x \in X(q)$ can be chosen as a function of $n \in \NN$ and $q \in \alpha$. As $\alpha$ is disjoint, we can therefore amalgamate the $x_{q, n} \in X(q)$ to an element $x_n \in X(p)$ such that
\[ p \forces \varphi(n, x_n). \]
So if we set $f(n) = x_n$ we obtain the desired result.
\end{proof}

\section{Main results}

In this final section we present the main results of this paper: the validity of various derived rules for {\bf CZF} and its extension of the form ${\bf CZF}^+$. A system of a slightly different kind to which these results apply as well will be discussed in Appendix B. The proofs are based on the fact that an appropriate predicative formulation of sheaf semantics can be proved inside {\bf CZF} to be sound for {\bf CZF}, together with the special features of the double construction mentioned after \reflemm{lemmaonforcing}.

\begin{theo}{derivedfanrule} {\rm (Derived Fan Rule)}
Suppose $\varphi(x)$ is a definable property of elements $u \in 2^{\lt \NN}$. If
\begin{eqnarray*}
{\bf CZF} & \vdash & (\forall \alpha \in 2^\NN) \, (\exists u \in 2^{\lt \NN}) \, (\alpha \in u \land \varphi(u)) \mbox{ and } \\
{\bf CZF} & \vdash &  (\forall u \in 2^{\lt \NN}) \, (\forall v \in 2^{\lt \NN}) \, (v \leq u \land \varphi(u) \to \varphi(v)),
\end{eqnarray*}
then ${\bf CZF} \vdash (\exists n \in \NN) \, (\forall v \in < \, >[n]) \, \varphi(v)$.
\end{theo}
\begin{proof}
We work in {\bf CZF}. We pass to sheaves over the double of formal Cantor space ${\cal D}(\mathbf{C})$, where there is a global section $\pi$ of the exponential sheaf $2^{\N}$ defined by letting $\pi(n)$ be the equivalence class of
\[  \big( \, < \, >[n], \lambda x \in < \, >[n]. x(n) \, \big) .\]
Under the correspondence between such global sections with continuous functions ${\cal D}(\mathbf{C}) \to \mathbf{C}$, this is precisely the map $\pi$ from Section 3.7 (second map in the list).

From
\[ {\rm Sh}({\cal D}(\mathbf{C})) \models  (\forall \alpha \in 2^\NN) \, (\exists u \in 2^{\lt \NN}) \, (\alpha \in u \land \varphi(u)), \]
it follows that
\[ D(< \, >) \forces  (\exists u \in 2^{\lt \NN}) \, (\pi \in u \land \varphi(u)). \]
Sheaf semantics then gives one a natural number $n$ such that for every $v \in < \, >[n]$ there is a section $\tau_v \in 2^{\lt \NN}(D(v))$ such that
\[ D(v) \forces  \pi \in \tau_v \land \varphi(\tau_v). \]
By choosing a larger $n$ if necessary, one may achieve that the $\tau_v$ are pure, i.e., of the form $(M_{v},  u_v)$. We will prove that this implies that $\varphi(v)$ holds.

From
\[ D(v) \forces  \pi \in \tau_v ,\]
it follows that $v \leq u_v$. Then validity of 
\[  (\forall u \in 2^{\lt \NN}) \, (\forall v \in 2^{\lt \NN}) \, (v \leq u \land \varphi(u) \to \varphi(v)) \] implies that $D(v) \forces \varphi(v)$. By picking a point $\alpha \in v$ and using the monotonicity of forcing, one gets $\{ \alpha \} \forces \varphi(v)$, and hence $\varphi(v)$ by part 3 of \reflemm{lemmaonforcing}.
\end{proof}

\begin{rema}{termextraction}
By using the fact that {\bf CZF} has the numerical existence property \cite{rathjen05} we see that the conclusion of the previous theorem could be strengthened to: then there is a natural number $n$ such that ${\bf CZF} \vdash (\forall v \in < \, >[n]) \, \varphi(v)$. Indeed, there is a primitive recursive algorithm for extracting this $n$ from a formal derivation in ${\bf CZF}$.
\end{rema}

\begin{rema}{derivedruleforCauchyreals}
It is not hard to show that {\bf CZF} proves the existence of a definable surjection $2^\N \to [0,1]_{Cauchy}$ from Cantor space to the set of Cauchy reals lying in the unit interval. This, in combination with \reftheo{derivedfanrule}, implies that one also has a derived local compactness rule for the Cauchy reals in {\bf CZF}. It also implies that we have a local compactness rule for the Dedekind reals in {\bf CZF + AC$_\omega$} and in {\bf CZF + DC}, because both {\bf AC}$_\omega$ and {\bf DC} are stable under sheaves over the double of formal Cantor space (see \refprop{preservationofchoice}) and using either of these two axioms, one can show that the Cauchy and Dedekind reals coincide.
\end{rema}

Recall that we use {\bf CZF$^+$} to denote any theory extending {\bf CZF} which allows one to prove set compactness and which is stable under sheaves.

\begin{theo}{derivedbarrule} {\rm (Derived Bar Induction Rule)} Suppose $\varphi(x)$ is a formula defining a subclass of $\NN^{\lt \NN}$. If
\begin{eqnarray*}
{\bf CZF}^+ & \vdash & (\forall \alpha \in \NN^\NN) \, (\exists u \in \NN^{\lt \NN}) \, (\alpha \in u \land \varphi(u)) \mbox{ and } \\
{\bf CZF}^+ & \vdash  &  (\forall u \in \NN^{\lt \NN}) \, (\forall v \in \NN^{\lt \NN}) \, (v \leq u \land \varphi(u) \to \varphi(v)) \mbox{ and } \\
{\bf CZF}^+ & \vdash & (\forall u \in \NN^{\lt \NN}) \, ((\forall n \in \NN) \, \varphi(u * n) \to \varphi(u)),
\end{eqnarray*}
then ${\bf CZF}^+ \vdash \varphi(< \, >)$.
\end{theo}
\begin{proof} We reason in {\bf CZF}$^+$. We pass to sheaves over the double of formal Baire space ${\cal D}(\mathbf{B})$, where there is a global section $\pi$ of the sheaf $\N^\N$ defined by letting $\pi(n)$ be the equivalence class of
\[ \big( \, < \, >[n], \lambda x \in < \, >[n]. x(n) \, \big) \]
(which corresponds to the ``projection'' ${\cal D}(\mathbf{B}) \to \mathbf{B}$, as before). From 
\[ {\rm Sh}({\cal D}(\mathbf{B})) \models (\forall \alpha \in \NN^\NN) \, (\exists u \in \NN^{\lt \NN}) \, (\alpha \in u \land \varphi(u)), \] one gets
\[ D(< \, >) \forces  (\exists u \in \NN^{\lt \NN}) \, (\pi \in u \land \varphi(u)). \]
By the sheaf semantics this means that there is a cover $S$ of $< \, >$ in formal Baire space $\mathbf{B}$ such that for every $v \in S$ there is a pure $u \in \NN^{\lt \NN}$ such that
\[ D(v) \forces  \pi \in u \land \varphi(u). \]
Now $ D(v) \forces  \pi \in u$ implies $v \leq u$ and because sheaf semantics is monotone this in turn implies $D(v) \forces \varphi(v)$. By choosing a point $\alpha \in v$ and using monotonicity again, one obtains that $\{ \alpha \} \forces \varphi(v)$, and hence $\varphi(v)$ by part 3 of \reflemm{lemmaonforcing}.

Summarising: we have a cover $S$ such that for all $v \in S$ the statement $\varphi(v)$ holds. Hence $\varphi(< \, >)$ holds by \refcoro{formalBIM}.
\end{proof}

\begin{theo}{derivedcontinuityrule} {\rm (Derived Continuity Rule for Baire Space)} Suppose $\varphi(x, y)$ is a formula defining a subset of $\NN^\NN \times \NN^\NN$. If ${\bf CZF}^+ \vdash (\forall \alpha \in \NN^\NN) \, (\exists ! \beta \in \NN^{\NN}) \, \varphi(\alpha, \beta),$ then
\[ {\bf CZF}^+ \vdash (\exists f: \NN^\NN \to \NN^\NN) \, [ \, ((\forall \alpha \in \NN^\NN) \, \varphi(\alpha, f(\alpha))) \land f \mbox{ continuous} \,]. \]
\end{theo}
\begin{proof} Again, we work in ${\bf CZF}^+$ and pass to sheaves over the double of formal Baire space ${\cal D}(\mathbf{B})$, where there is a global section of the sheaf $\N^\N$, namely the projection $\pi: {\cal D}(\mathbf{B}) \to\mathbf{B}$. Since
\[ {\rm Sh}{\cal (D}(\mathbf{B})) \models (\exists ! \beta \in \NN^\NN) \, \varphi(\pi, \beta), \]
there exists a unique function $\rho: {\cal D}(\mathbf{B}) \to \mathbf{B}$ (and global section of $\NN^\NN$) such that
\[ D(< \, >) \vdash \varphi(\pi, \rho). \]
Consider the maps $\mu: \mathbf{B} \to {\cal D}(\mathbf{B})$ and $\nu: \mathbf{B}_{discr} \to {\cal D}(\mathbf{B})$ from Section 3.7. The continuity of $\rho$ implies that ${\rm pt}(\rho\mu) = {\rm pt}(\rho \nu): \N^\N \to \N^\N$; writing $f = {\rm pt}(\rho\mu)$, one sees that $f: \N^\N \to \N^\N$ is continuous. Moreover, if $\alpha \in \NN^\NN$, then $\{\alpha \} \forces \varphi({\rm pt}(\pi)(\alpha), {\rm pt}(\rho)( \alpha ))$, i.e.~$\{ \alpha \} \forces \varphi(\alpha, f(\alpha))$, and hence $\varphi(\alpha, f(\alpha))$.
\end{proof}

These proofs can be adapted in various ways to prove similar results for  (extensions of) {\bf CZF}, for instance:
\begin{itemize}
\item[-] \reftheo{derivedfanrule} holds for any extension of {\bf CZF} which is stable under sheaves over the double of formal Cantor space, such as the extension of {\bf CZF} with choice principles like {\bf DC} or ${\bf AC}_\omega$ (because of \refprop{preservationofchoice}).
\item[-] For the same reason \reftheo{derivedbarrule} and \reftheo{derivedcontinuityrule} remain valid if we extend ${\bf CZF}^+$ with choice principles. These results also hold for the theory {\bf CZF} +  ${\bf AC}_\omega$ + ``The Brouwer ordinals form a set'' (see Appendix B).
\item[-] The same method of proof as in \reftheo{derivedcontinuityrule} should establish a derived continuity rule for the Dedekind reals and many other definable formal spaces.
\end{itemize}

\appendix

\section{Independence of presentability of formal Baire space in CZF}

The aim of this appendix is to show that {\bf CZF} does not prove the existence of a presentation of formal Baire space. For this purpose, we use forcing over a site as in \cite{bergmoerdijk10b} rather than forcing over a formal space. Recall that in \cite{bergmoerdijk10b} we showed that what happens when one does forcing over a site is completely analogous to what happens when one does forcing over a formal space: it leads, provably in {\bf CZF}, to a sound semantics of {\bf CZF}, as long as the site is assumed to have a presentation (see \cite{bergmoerdijk10b} for the definition of a presentation for a site).

Let $\mathbb{S}$ be a small category of formal spaces and continuous maps, whose objects are basic open subsets of $\mathbf{B}^n$ of the form $B(u_1) \times \ldots \times B(u_n)$ (where $u_1, \ldots, u_n \in \N^{\lt \N}$ and we write $B(u)$ for the basic open determined by the finite sequence $u$) and whose maps contain the inclusions between open subsets and the projections. (For example, $\mathbb{S}$ could be given by these objects and \emph{all} continuous maps between them.) Equip $\mathbb{S}$ with the Grothendieck topology induced by the open covers of the formal space $\mathbb{S}$ and the projections (the latter are automatically included if, for example, constant maps are included so that projections have sections in $\mathbb{S}$).

\begin{lemm}{BpresSpres}
If $\mathbf{B}$ has a presentation, then so does $\mathbb{S}$.
\end{lemm}
\begin{proof}
This is clear from the fact that if $\mathbf{B}$ has a presentation, so does each $B(u_1) \times \ldots \times B(u_n)$. In fact, to be explicit, for $X = B(u_1) \times \ldots \times B(u_n)$ the formula
\[ {\rm BCov}_{\mathbb{S}}(X) = \big\{ \, \{ U_i \times Y_i \to X \} \, : \, Y_i \in \mathbb{S}, \{ U_i \}  \in {\rm BCov}_{B(u_1) \times \ldots \times B(u_n)}(X) \, \big\} \]
defines a presentation of $\mathbb{S}$.
\end{proof}

By the lemma it follows that if $\mathbf{B}$ has a presentation, the sheaves on the site $\mathbb{S}$ provide a model for ${\bf CZF}$. We observe the following property of the model:

\begin{prop}{BIissheavesoverS} {\rm (See \cite{fourman84} and \cite[Section 15.6]{troelstravandalen88b}.)} Assume $\mathbf{B}$ has a presentation. Then Monotone Bar Induction holds in the {\bf CZF}-model given by sheaves on $\mathbb{S}$.
\end{prop}

\begin{coro}{BhasnopresinCZF}
The theory {\bf CZF + DC} does not prove that $\mathbf{B}$ has a presentation.
\end{coro}
\begin{proof}
If {\bf CZF} + {\bf DC} would prove that $\mathbf{B}$ has a presentation, then, by the proposition, this would imply that the consistency of {\bf CZF} + {\bf DC} implies the consistency of {\bf CZF} + Monotone Bar Induction. But the latter is known to have greater proof-theoretic strength (see \cite{rathjen06c}).
\end{proof}

\begin{proof} (Of \refprop{BIissheavesoverS}.) Suppose $X \in \mathbb{S}$ and $S \in {\rm Pow}(\N^{\lt \N})(X)$ is a (small) subsheaf of $\N^{\lt \N}$ which forms an ``internal bar''; i.e.,
\begin{enumerate}
\item[(1)] $X \forces (\forall \alpha \in \N^\N) \, (\exists n \in \N) \, (\alpha(0), \ldots, \alpha(n)) \in S$, 
\item[(2)] $X \forces (\forall u, v \in \N^{\lt \N}) \, (u \leq v \land v \in S \to u \in S)$,
\item[(3)] $X \forces (\forall u \in \N^{\lt \N}) \, ( (\forall n \in N) u * <n> \in S \to u \in S)$.
\end{enumerate}
The projection $\pi_2: X \times \mathbf{B} \to\mathbf{B}$ at the stage $\pi_1: X \times \mathbf{B} \to X$ over $X$ represents a (generic) element of the sheaf $\N^\N$, and (1) implies
\[ X \times \mathbf{B} \forces (\exists n \in \N) \, (\pi_2(0), \ldots, \pi_2(n)) \in S. \]
By definition, this means that there is a cover of $X \times \mathbf{B}$ by basic opens $U_i \times B(v_i)$ such that for each $i$ we have that
\begin{enumerate}
\item[(4)] $(\exists n \in \NN) \, U_i \times B(v_i) \forces (\pi_2(0), \ldots, \pi_2(n)) \in S$
\end{enumerate}
(where we simply write $S$ for the restriction of $S$ along $U_i \times B(v_i) \to U_i \subseteq X$). We claim that we can choose the cover $U_i \times B(v_i)$ in such a way that for each $i$ it holds that
\begin{enumerate}
\item[(5)] $U_i \forces v_i \in S$.
\end{enumerate}
Indeed, if we can choose the $n$ in (4) such that $n \leq |v_i|$, then (4) implies (5) by assumption (2) on $S$. On the other hand, if (4) holds for $n \gt |v_i|$, we can replace the single element $U_i \times B(v_i)$ in the cover by all elements of the form $U_i \times B(w)$ where $w$ is an extension of $v_i$ of length $n$. Then $U_i \times B(w) \forces w \in S$ by (4) and monotonicity of forcing, hence $U_i \forces w \in S$ because projections cover.

Let
\[ {\cal W} = \big\{ \, U \times B(v) \in \mathbb{S} \, : \, U \subseteq X \mbox{ open}, v \in \N^{\lt \N} \mbox{ and } U \forces v \in S \,  \big\}. \]
Then ${\cal W}$ covers $X \times \mathbf{B}$ as we have just seen. Moreover, if any $U \times B(v)$ is covered by elements of ${\cal W}$ then it belongs to ${\cal W}$. Indeed, to show this it suffices to prove the following two properties:
\begin{enumerate}
\item[(6)] If $\{ U_i \}$ covers $U$ and $U_i \times B(v) \in {\cal W}$, then $U \times B(v) \in {\cal W}$.
\item[(7)] If $U \times B(v * <n>) \in {\cal W}$ for each $n$, then $U \times B(v) \in {\cal W}$.
\end{enumerate}
But (6) holds by the local character of forcing, while (7) holds by assumption (3) on $S$. By induction on covers (\reftheo{indoncovers}) we conclude that $X \forces <> \in S$, which completes the proof.
\end{proof}

\section{Brouwer ordinals}

Recall that we defined ${\bf CZF}^+$ to be any extension of {\bf CZF} in which the set compactness theorem is provable and which is stable under sheaves. We do not expect that {\bf CZF} + {\bf AC}$_\omega$ + ``The Brouwer ordinals form a set'' is such a theory ${\bf CZF}^+$. Nevertheless, our main results apply to this theory as well. To show this, we have to prove (1) that this theory proves that formal Baire space is presentable and (2) that this theory is stable under taking sheaves over the double of formal Baire space. In this appendix we work out the details.

First, we recall the definition of the Brouwer ordinals.
\begin{defi}{brouwerordinals}
The class $BO$ of \emph{Brouwer ordinals} is the smallest class closed under the rules:
\begin{displaymath}
\begin{array}{c}
* \in BO, \\
t: \N \to BO \Rightarrow {\rm sup}(t) \in BO.
\end{array}
\end{displaymath}
In other words, it is the W-type associated to the constant map $\NN \to 2$ with value 1 or the initial algebra for the functor $F(X) = 1 + X^\N$ (see \cite{moerdijkpalmgren00}).
\end{defi}

Our proof that the theory {\bf CZF} + {\bf AC}$_\omega$ + ``The Brouwer ordinals form a set'' shows that formal Baire space has a presentation, is based on an alternative description of formal Baire space. For this, define ${\rm BCov}(< \, >)$ be smallest subclass of ${\rm Pow}(\NN^{\lt \NN})$ such that:
\begin{displaymath}
\begin{array}{l}
\{ < \, > \} \in {\rm BCov}(< \, >) \\
\forall i \in \NN: \, S_i \in {\rm BCov}(< \, >)  \Rightarrow \bigcup_{i \in \N} <i> * S_i \in {\rm BCov}(< \, >)
\end{array}
\end{displaymath}
This inductive definition makes sense in {\bf CZF} even when the Brouwer ordinals only form a class. 

\begin{lemm}{BOsetimpliesBCovset}
If $BO$ is a set, then so is ${\rm BCov}(< \, >)$.
\end{lemm}
\begin{proof}
Define a map $k: BO \to {\rm Pow}(\N^{\lt \N})$ by recursion:
\begin{eqnarray*}
k(*) & = & \{ < \, > \}, \\
k({\rm sup}(t)) & = & \bigcup_{i \in \N} <i> * t(i).
\end{eqnarray*}
Its image is ${\rm BCov}(<>)$ and therefore it follows from replacement that it is a set, if $BO$ is a set.
\end{proof}

Put:
\begin{eqnarray*}
S \in {\rm BCov}(u)  & \Leftrightarrow & \exists T \in {\rm BCov}(< \, >): u * T \in {\rm BCov}(u)  \\
S \in {\rm Cov}(u) & \Leftrightarrow & \exists T \in {\rm BCov}(u): T \subseteq S.
\end{eqnarray*}

\begin{lemm}{hilflemma}
\begin{enumerate}
\item Every $T \in {\rm BCov}(u)$ is countable.
\item Suppose $R_v \in {\rm BCov}(v)$ is a collection of basic covering sieves indexed by elements $v$ from a sieve $T$.  If $T$ belongs to ${\rm BCov}(u) $ then so does $\bigcup_{v \in T} R_v$.
\item If $T \in {\rm BCov}(u)$ and $v \leq u$, then there is an $S \in {\rm BCov}(v)$ such that $S \subseteq v^* \downarrow T$.
\end{enumerate}
\end{lemm}
\begin{proof}
It suffices to prove these statements in the special case where $u = < \, >$; in that case, they follow easily by induction on $T$.
\end{proof}

\begin{prop}{BSassiteCAC} {\rm ({\rm ${\bf AC}_\omega$}) } $(\NN^{\lt \NN}, {\rm Cov})$ as defined above is an alternative description of formal Baire space and therefore formal Baire space is presentable whenever the Brouwer ordinals form a set.
\end{prop}
\begin{proof} We begin by showing that $(\NN^{\lt \NN}, {\rm Cov})$ is a formal space. Since maximality is clear and stability follows from item 3 of the previous lemma, it remains to check local character.

Suppose $S$ is a sieve on $u$ and there is a sieve $R \in {\rm Cov}(u)$ such that for all $v \in R$ the sieve $v^*S$ belongs to ${\rm Cov}(v)$. Since $R \in {\rm Cov}(u)$ there is a $T \in {\rm BCov}(u)$ such that $T \subseteq R$. Therefore we have for any $v \in T$ that $v^*S$ covers $v$ and hence that there is a $Z \in {\rm BCov}(v)$ such that $Z \subseteq S$. Since $T$ is countable, we can use ${\bf AC}_\omega$ or the finite axiom of choice (which is provable in {\bf CZF}) to choose the elements $Z$ as a function $Z_v$ of $v \in T$. Then let $K =\bigcup_{v \in T} Z_v$. $K$ is covering by the previous lemma and because
\[ K =  \bigcup_{v \in T} Z_v \subseteq S, \]
the same must be true for $S$. 

It now follows from the discussion preceeding the definition of ${\rm BCov}(u)$ above that each ${\rm BCov}(u)$ will be a set, if the Brouwer ordinals form a set. Hence the formal space $(\NN^{\lt \NN}, {\rm Cov})$ will be presentable if $BO$ is a set.

To easiest way to prove that we have given a different presentation of formal Baire space is to show that Cov is the smallest topology such that
\[ \downarrow \{ u * <n> \, : \, n \in \N \}  \in {\rm Cov}(u). \]
Clearly, Cov has this property, so suppose ${\rm Cov}^*$ is another. One now shows by induction on $T \in {\rm BCov}(u)$ that $\downarrow T \in {\rm Cov}^*(u)$. This completes the proof.
\end{proof}

\begin{coro}{BOsetindependent}
The theory ${\bf CZF} + {\bf DC}$ does not prove that the Brouwer ordinals form a set.
\end{coro}
\begin{proof}
Because this theory does not prove that formal Baire space has a presentation (see \refcoro{BhasnopresinCZF}).
\end{proof}

It will follow from the next theorem, whose proof will take the remainder of this appendix, that {\bf CZF} + {\bf AC}$_\omega$ + ``The Brouwer ordinals form a set'' is a theory which is stable under taking sheaves over the double of formal Baire space.

\begin{theo}{preservationofchoiceandbrouwerord} Let $(\mathbb{P}, {\rm Cov})$ be a presentable formal space which is CC. Then the combination of ${\bf AC}_\omega$ and smallness of the Brouwer ordinals implies their joint validity in ${\rm Sh}(\mathbb{P}, {\rm Cov})$. 
\end{theo}

In view of \refprop{preservationofchoice} it suffices to show that the Brouwer ordinals are small in ${\rm Sh}(\mathbb{P}, {\rm Cov})$. To that purpose, we will give an explicit construction of the Brouwer ordinals in this category, from which it can immediately be seen that they are small (the description is a variation on those presented in \cite{bergmoerdijk08b} and \cite{bergmoerdijk10b}).

Let ${\cal V}$ be the class of all well-founded trees, in which 
\begin{itemize}
\item nodes are labelled with triples $(p, \alpha, \varphi)$ with $p$ an element of $\mathbb{P}$, $\alpha$ a countable and disjoint subset of $\downarrow p$ such that $\downarrow \alpha \in {\rm Cov}(p)$ and $\varphi$ a function $\alpha \to \{ 0, 1 \}$,
\item edges into nodes labelled with $(p, \alpha, \varphi)$ are labelled with pairs $(q, n)$ with $q \in \alpha$ and $n \in \N$,
\end{itemize}
in such a way that
\begin{itemize}
\item if a node is labelled with $(p, \alpha, \varphi)$ and $q \in \alpha$ is such that $\varphi(q) = 0$, then there is no edge labelled with $(q, n)$ into this node, but
\item if a node is labelled with $(p, \alpha, \varphi)$ and $q \in \alpha$ is such that $\varphi(q)=1$, then there is for every $n \in \N$ a unique edge into this node labelled with $(q, n)$.
\end{itemize}
Using that the Brouwer ordinals form a set, one can show also that ${\cal V}$ is a set. If $v$ denotes a well-founded tree in ${\cal V}$, we will also use the letter $v$ for the function that assigns to labels of edges into the root of $v$ the tree attached to this edge. So if $(q, n)$ is a label of one of the edges into the root of $v$, we will write $v(q, n)$ for the tree that is attached to this edge; this is again an element of ${\cal V}$. Note that an element of  ${\cal V}$ is uniquely determined by the label of its root and the function we just described. 

We introduce some terminology and notation: we say that a tree $v \in {\cal V}$ is \emph{rooted} at an element $p$ in $\mathbb{P}$, if its root has a label whose first component is $p$. A tree $v \in {\cal V}$ whose root is labelled with $(p, \alpha, \varphi)$ is \emph{composable}, if for any $(q, n)$ with $q \in \alpha$ and $\varphi(q) =1$, the tree $v(q, n)$ is rooted at $q$. We will write ${\cal W}$ for the set of trees that are \emph{hereditarily} composable (i.e. not only are they themselves composable, but the same is true for all their subtrees). 

Next, we define by transfinite recursion a relation $\sim$ on ${\cal V}$:
\begin{center}
\begin{tabular}{lcp{8 cm}}
$v \sim v'$ & $\Leftrightarrow$ & If the root of $v$ is labelled with $(p, \alpha, \varphi)$ and the root of $v'$ with $(p', \alpha', \varphi')$, then $p = p'$ and $p$ is covered by those $r \leq p$ for which there are (necessarily unique) $q \in \alpha$ and $q' \in \alpha'$ such that (1) $r \leq q$ and $r \leq q'$, (2) $\varphi(q) = \varphi'(q')$ and (3) $\varphi(q) = \varphi'(q') =1$ implies $v(q, n) \sim v'(q', n)$ for all $n \in \N$.
\end{tabular}
\end{center}
By transfinite induction one verifies that $\sim$ is an equivalence relation on both ${\cal V}$ and ${\cal W}$. Write ${\cal BO}$ for the quotient of ${\cal W}$ by $\sim$. The following sequence of lemmas establishes that ${\cal BO}$ can be given the structure of a sheaf and is in fact the object of Brouwer ordinals in the category of sheaves.

\begin{lemm}{Wbarpresheaf}
${\cal BO}$ can be given the structure of a presheaf.
\end{lemm}
\begin{proof}
Since by definition of $\sim$, all trees $w \in {\cal W}$ in an equivalence class are rooted at the same element, we can say without any danger of ambiguity that an element $\overline{w} \in {\cal BO}$ is rooted at $p \in \mathbb{P}$. We will denote the collection of trees in ${\cal BO}$ rooted at $p$ by ${\cal BO}(p)$. 

Suppose $[w] \in {\cal W}(p)$ and $q \leq p$. If the root of $w$ is labelled by $(p, \alpha, \varphi)$, then there is a countable and disjoint refinement $\beta$ of $q^*\downarrow \alpha$ (by stability and the fact that $(\mathbb{P}, {\rm Cov})$ is a CC-space). For each $r \in \beta$ there is a unique $q \in \alpha$ such that $r \leq q$ (by disjointness), so one can define $\psi: \beta \to \{0, 1\}$ by $\psi(r) = \varphi(q)$ and, whenever $\psi(r) = \varphi(q) = 1$, $v(r, n) = w(q, n)$. The data $(q, \beta, \psi)$ and $v$ determine an element $w' \in {\cal W}(q)$ and we put
\[ [w] \upharpoonright q = [w']. \]
One easily verifies that this is well-defined and gives ${\cal BO}$ the structure of a presheaf.
\end{proof}

\begin{lemm}{Wbarseparated}
${\cal BO}$ is separated.
\end{lemm}
\begin{proof} 
Suppose $T$ is a sieve covering $p$ and $w, w' \in {\cal W}(p)$ are such that $[w] \upharpoonright t = [w'] \upharpoonright t$ for all $t \in T$. We have to show $w \sim w'$, so suppose $(p, \alpha, \varphi)$ is the label of the root of $w$ and $(p', \alpha', \varphi')$ is the label of the root of $w'$. Since $w'$ is rooted at $p'$, we have $p = p'$.

Let $R$ consist of those $r  \in (\downarrow\alpha) \cap (\downarrow \alpha')$, for which there are $q \in \alpha$ and $q' \in \alpha'$ such that (1) $r \leq q, q'$, (2) $\varphi(q) = \varphi'(q')$ and (3) $\varphi(q) = \varphi'(q') = 1$ implies $w(q, n) \sim w'(q', n)$ for all $n \in \N$. $R$ is a sieve, and the statement of the lemma will follow once we show that it is covering.

Fix an element $t \in T$. Unwinding the definitions in $[w] \upharpoonright t = [w'] \upharpoonright t$ gives us the existence of a covering sieve $S \subseteq t^*(\downarrow \alpha) \cap t^*(\downarrow \alpha')$ such that $S \subseteq t^*R$. So $R$ is a covering sieve by local character.
\end{proof}

\begin{lemm}{Wbarsheaf}
${\cal BO}$ is a sheaf.
\end{lemm}
\begin{proof} Let $S$ be a covering sieve on $p$ and suppose we have a compatible family of elements $(\overline{w}_q \in {\cal BO})_{q \in S}$. Let $\alpha$ be a countable and disjoint refinement of $S$ and use ${\bf AC}_\omega$ to choose for every element $q \in \alpha$ a representative $(w_q \in {\cal W})_{q \in \alpha}$ such that $[w_q] = \overline{w}_q$. For every $q \in \alpha$ the representative $w_q$ has a root labelled by something of form $(q, \beta_q, \varphi_q)$. If we put $\beta = \bigcup_{q \in \alpha} \beta_q$, then $\beta$ is countable and disjoint and $\downarrow \beta$ covers $p$ (by local character). If $r \in \beta$, then there is a unique $q \in \alpha$ such that $r \in \beta_q$ (by disjointness), so therefore it makes sense to define $\varphi(r) = \varphi_q(r)$ and $w(r, n) = w_q(r, n)$. 

We claim the element $[w] \in {\cal BO}$ determine by the data $(p, \beta, \varphi)$ and the function $w$ just defined is the amalgamation of the elements $(\overline{w}_q \in {\cal BO})_{q \in S}$. To that purpose, it suffices to prove that $[w] \upharpoonright q = \overline{w}_q = [w_q]$ for all $q \in \alpha$. This is not hard, because if $q \in \alpha$ and $r \in \beta_q$, then $w(r, n) = w_q(r, n)$, by construction. This completes the proof.
\end{proof}

\begin{lemm}{Wbaralgebra}
${\cal BO}$ is an algebra for the functor $F(X) = 1 + X^\N$.
\end{lemm}
\begin{proof} We have to describe a natural transformation ${\rm sup}: F({\cal BO}) \to {\cal BO}$, i.e., a natural transformation $1 \to {\cal BO}$ and a natural transformation ${\cal BO}^{\N} \to {\cal BO}$. For the natural transformation $1 \to {\cal BO}$, we define ${\rm sup}_p(*)$ on $* \in 1(p)$ to be the equivalence class of the unique element in ${\cal W}$ determined by the data $(p, \{ p \}, \varphi)$ with $\varphi(p) = 0$. To define the natural transformation ${\cal BO}^{\N} \to {\cal BO}$ on $\overline{t}: \N \to {\cal BO}(p)$, we use ${\bf AC}_\omega$ to choose a function $t: \N \to {\cal W}(p)$ such that $[t(n)] = \overline{t}(n)$ for all $n \in \N$. Then we define ${\rm sup}_p(\overline{t})$ to be the equivalence class of the element $w$ determined by the data $(p, \{ p \}, \varphi)$ with $\varphi(p) = 1$ and $w(p, n) = t(n)$. We leave the verification that this makes ${\rm sup}$ well-defined and natural to the reader.
\end{proof}

\begin{lemm}{Wbarinitialalgebra}
${\cal BO}$ is the initial algebra for the functor $F(X) = 1 + X^\N$.
\end{lemm}
\begin{proof}
We follow the usual strategy: we show that ${\rm sup}: F({\cal BO}) \to {\cal BO}$ is monic and that ${\cal BO}$ has no proper $F$-subalgebras (i.e., we apply Theorem 26 of \cite{berg05}). It is straightforward to check that sup is monic, so we only show that ${\cal BO}$ has no proper $F$-subalgebras, for which we use the inductive properties of ${\cal V}$. 

Let $I$ be a sheaf and $F$-subalgebra of ${\cal BO}$. We claim that 
\[ J = \{ v \in {\cal V} \, : \, \mbox{if } v \mbox{ is hereditarily composable, then } [v] \in I \} \]
is such that if all immediate subtrees of an element $v \in {\cal V}$ belong to it, then so does $v$ itself. 

Proof: Suppose $v \in {\cal V}$ is a hereditarily composable tree such that all its immediate subtrees belong to $J$. Assume moreover that $(p, \alpha, \varphi)$ is the label of its root. We know that for all $n \in \N$ and $q \in \alpha$ with $\varphi(q) = 1$, $[v(f, y)] \in I$ and our aim is to show that $[v] \in I$.

For the moment fix an element $q \in \alpha$. Either $\varphi(q) = 0$ or $\varphi(q) = 1$. If $\varphi(q) = 0$, then $[v] \upharpoonright q$ equals ${\rm sup}_q(*)$ and therefore $[v] \upharpoonright q \in I$, because $I$ is a $F$-algebra. If $\varphi(q) = 1$, then we may put $\overline{t}(n) = [v(q, n)]$ and $[v] \upharpoonright q$ will equal ${\rm sup}_q(\overline{t})$. Therefore  $[v] \upharpoonright q \in I$, again because $J$ is a $F$-algebra. So for all $q \in \alpha$ we have $[v] \upharpoonright q \in I$. But then it follows that $[v] \in I$, since $I$ is a sheaf.

We conclude that $J = {\cal V}$ and $I = {\cal BO}$.
\end{proof}

This completes the proof of the correctness of our description of the Brouwer ordinals in a category of sheaves and thereby of \reftheo{preservationofchoiceandbrouwerord}.

\bibliographystyle{plain} \bibliography{ast}

\end{document}